\documentclass[11pt]{amsart}

\usepackage[utf8]{inputenc}  

\usepackage{amsthm}
\usepackage{amsfonts}
\usepackage{amsmath}
\usepackage{amssymb}
\usepackage{mathtools} 
\usepackage{stmaryrd}
\usepackage{marvosym}

\usepackage{xcolor}
\definecolor{darkgreen}{rgb}{0,0.45,0}
\definecolor{darkred}{rgb}{0.75,0,0}
\definecolor{darkblue}{rgb}{0,0,0.6}
\usepackage[colorlinks,citecolor=darkgreen,linkcolor=darkred,urlcolor=darkblue]{hyperref}
\usepackage{url}
\usepackage{breakurl}
\usepackage[mathscr]{eucal}
\usepackage{enumerate} 
\usepackage[capitalize]{cleveref}

\usepackage{tikz}
\usetikzlibrary{arrows}
\tikzset
{
  diagram/.style=
  {
    matrix of math nodes,
    column sep=2.5em,
    row sep=2.5em,
    text height=1.5ex,
    text depth=.25ex
  },
  cross line/.style={preaction={draw=white,-,line width=6pt}},
  every to/.style={font=\footnotesize},
  cof/.style={>->},                
  fib/.style={->>},                
  inj/.style={right hook->},       
  surj/.style={-open triangle 60}, 
  eq/.style={-,double distance=1.7pt}
}

\makeatletter
\newenvironment{ctikzpicture}
{
  \begingroup
  \smallskip
  \par
  \centering
  \begin{tikzpicture}
}{
  \end{tikzpicture}
  \par
  \smallskip
  \endgroup
  \aftergroup\@afterindentfalse
  \aftergroup\@afterheading
}
\makeatother

\theoremstyle{plain}
\newtheorem{theorem}{Theorem}[section]

\newtheorem{proposition}[theorem]{Proposition}
\newtheorem{lemma}[theorem]{Lemma}

\theoremstyle{definition}
\newtheorem{definition}[theorem]{Definition}

\newtheorem{example}[theorem]{Example}
\newtheorem{examples}[theorem]{Examples}

\newtheorem{remark}[theorem]{Remark}

\newtheorem{construction}[theorem]{Construction}

\Crefname{corollary}{Corollary}{Corollaries}
\Crefname{theorem}{Theorem}{Theorems}
\Crefname{proposition}{Proposition}{Propositions}

\usetikzlibrary{matrix}

\tikzset
{
  diagram/.style=
  {
    matrix of math nodes,
    column sep=2.5em,
    row sep=2.5em,
    text height=1.5ex,
    text depth=.25ex
  },
  every to/.style={font=\footnotesize},
}
\makeatletter
\renewcommand{\paragraph}{\@startsection{paragraph}{4}{0mm}{-0.5\baselineskip}{-1ex}{\bf}}
\makeatother

\newcommand{\Cat}{\mathsf{Cat}} 
\newcommand{\hoCat}{\mathsf{hoCat}} 
\newcommand{\qCat}{\mathsf{qCat}} 
\newcommand{\Set}{\mathsf{Set}} 
\newcommand{\sSet}{\mathsf{sSet}} 
\newcommand{\weCat}{\mathsf{weCat}}

\newcommand{\C}{\mathscr{C}} 
\newcommand{\D}{\mathscr{D}} 
\newcommand{\M}{\mathscr{M}} 

\renewcommand{\L}{\mathrm{L}} 
\newcommand{\N}{\mathrm{N}} 
\newcommand{\Nf}{\mathrm{N}_\mathrm{f}} 

\newcommand{\cod}{\mathrm{cod}} 
\newcommand{\ev}{\mathrm{ev}}  
\newcommand{\Ho}{\mathrm{Ho}} 
\newcommand{\Hom}{\operatorname{Hom}} 
\newcommand{\ob}{\operatorname{Ob}} 
\newcommand{\op}{\mathrm{op}} 
\newcommand{\Sd}{\mathrm{Sd}} 

\newcommand{\bbN}{\mathbb{N}} 

\newcommand{\U}{\mathrm{U}} 

\newcommand{\adjoint}{\dashv} 

\renewcommand{\to}{\rightarrow} 
\newcommand{\we}{\sim}
\newcommand{\into}{\hookrightarrow} 
\newcommand{\iso}{\cong} 
\newcommand{\fib}{\twoheadrightarrow} 
\newcommand{\slice}{\! \downarrow \!}
\newcommand{\fibslice}{\!\! \mathrel{\rotatebox[origin=c]{270}{$\twoheadrightarrow$}}\!\! }

\newcommand{\mC}{\mathsf{C}}
\renewcommand{\U}{\mathrm{U}}
\newcommand{\Id}{\mathrm{Id}}
\newcommand{\Cxl}{\mathsf{Cxl}_{\Pi,\Sigma,\Id}} 
\newcommand{\cxt}{\mathsf{cxt}}
\newcommand{\ft}{\mathrm{ft}}

\newcommand{\tA}{\widetilde{A}}
\newcommand{\tB}{\widetilde{B}}
\newcommand{\tX}{\widetilde{X}}
\newcommand{\tY}{\widetilde{Y}}
\renewcommand{\S}{\mathcal{S}}

\newcommand{\w}{\mathrm{w}}

\begin{document}
\title{Locally cartesian closed quasicategories from type theory}
\author{Krzysztof Kapulkin}
\date{\today}
\subjclass[2010]{18G55 (primary); 03B70, 55U35 (secondary)}
\begin{abstract}
We prove that the quasicategories arising from models of Martin-L\"of type theory via simplicial localization are locally cartesian closed.
\end{abstract}
\maketitle

\section*{Introduction}\label{sec:introduction}

Intensional Martin-L\"of Type Theory with $\mathsf{\Pi}$-, $\mathsf{\Sigma}$-, and $\mathsf{Id}$-types \cite{martin-lof:bibliopolis} has been conjectured to be the internal language of locally cartesian closed quasicategories. A precise form of this conjecture would consist of giving a pair of functors between categories of suitable type theories and locally cartesian closed quasicategories, and showing that they constitute an equivalence (in the appropriate homotopical sense).

While the conjecture seems very natural to someone familiar with the ideas of Homotopy Type Theory, it has proven difficult, and at the moment even the constructions of the functors between type theories and quasicategories are known only partially. More precisely, using the ideas and results of Lumsdaine--Warren \cite{lumsdaine-warren:local-universes} and Gepner--Kock \cite{gepner-kock:univalence}, Cisinski and Shulman \cite{shulman:inverse-diagrams} showed that every locally presentable locally cartesian closed quasicategory can be presented by a model category admitting a model of type theory. The assumption of local presentability is essential in their proof.

The main contribution of the present paper is the other direction of the proposed correspondence. Precisely, we prove that the simplicial localization of any categorical model of type theory (that is, a contextual category with $\mathsf{\Pi}$-, $\mathsf{\Sigma}$-, and $\mathsf{Id}$-structure) regarded as a category with weak equivalences is a locally cartesian closed quasicategory.

The essential idea of the proof is the use of the construction of the quasicategory of frames, introduced by Szumi{\l}o \cite{szumilo:two-models}, allowing for an explicit description of the simplicial localization for a wide class of categories with weak equivalences, namely those that carry the structure of a (co)fibration category \cite{kapulkin-szumilo}. It is indeed the case that every categorical model of type theory can naturally be equipped with such a structure \cite{avigad-kapulkin-lumsdaine}, \cite{shulman:inverse-diagrams}.

The paper is organized as follows. In \Cref{sec:type-theory}, we review the background on models of type theory and abstract homotopy theory, and formulate our main theorem. In \Cref{sec:fib-cat}, we then gather several notions and lemmas concerning fibration categories and the construction of the quasicategory of frames. The two crucial properties of this construction are established in \Cref{sec:slices,sec:adjoints}, where we prove that, in a suitable sense, taking quasicategories of frames preserves slicing and adjoints, respectively. These results are then used in \Cref{sec:the-proof} for the proof of our main theorem.

\textbf{Acknowledgements.} This result would never have been possible without Karol Szumi{\l}o. I would like to thank him for sharing early drafts of his thesis with me, many illuminating discussions, and helpful comments on an early draft of this paper. I am also indebted to Peter LeFanu Lumsdaine, from whom I learned a lot about type theory and its categorical models, as well as about English grammar (in particular, where not, to put a comma). Last, but not least, I would like to thank my advisor, Tom Hales, for his patience and guidance and for making me the mathematician I am today. This paper is dedicated to my mother.

\section{Background and statement of the main theorem}\label{sec:type-theory}

In this section, we will review the necessary background, leading to the statement of our main theorem. The section is therefore divided into three subsections: the first two cover the background on categorical models of type theory (\Cref{sec:models-of-tt}) and abstract homotopy theory (\Cref{sec:aht}), respectively; and the last one contains the statement of our main result (\Cref{sec:statement}).

\subsection{Categorical models of type theory}\label{sec:models-of-tt} 

While the title of the paper suggests that we extract locally cartesian closed quasicategories from type theory, we will not work explicitly with syntactically presented type theories. Instead, we shall phrase all our results in the language of contextual categories equipped with additional structure corresponding to the logical constructors $\mathsf{\Pi}$, $\mathsf{\Sigma}$, and $\mathsf{Id}$ of type theory (see \cite{cartmell:thesis}, \cite{cartmell:generalised-algebraic-theories}, and \cite{streicher:semantics-book} for the relationship between these and syntactically presented type theories; and \cite[Sec.\ 2.2.1]{avigad-kapulkin-lumsdaine} for the statement of the conjectured correspondence). We begin by recalling the definition of a contextual category:

\begin{definition}
 A \textbf{contextual category} $\mC$ consists of the following data:
 \begin{enumerate}
  \item a small category, also denoted $\mC$, together with a grading on objects $\ob \mC = \coprod\limits_{n\in \bbN} \ob_n \mC$;
  \item an object $\diamond \in \ob_0\mC$;
  \item for each $n \in \bbN$, a map $\ft_n \colon \ob_{n+1}\mC \to \ob_n \mC$;
  \item for each $n \in \bbN$ and each $X \in \ob_{n+1}\mC$, a map $p_X \colon X \to \ft X$;
  \item for each $n \in \bbN$, each $X \in \ob_{n+1}\mC$, and each $f \colon Y \to \ft X$, an object $f^*X$ together with a map $q_{f,X} \colon f^* X \to X$;
 \end{enumerate}
 subject to the following conditions:
 \begin{enumerate}
 \item $\diamond$ is the unique element of $\ob_0\mC$;
 \item $\diamond$ is terminal in $\mC$;
 \item for each $n \in \bbN$, each $X \in \ob_{n+1}\mC$, and each $f \colon Y \to \ft X$, $\ft f^*X = Y$ and the following square is a pullback:
 \begin{ctikzpicture}
   \matrix[diagram]
   {
     |(fX)| f^*X & |(X)| X \\
     |(Y)| Y   & |(ftX)| \ft X \\
   };

   \draw[->] (fX) to node[above] {$q_{f,X}$} (X);

   \draw[->] (X) to node[right] {$p_X$} (ftX);

   \draw[->] (Y) to node[above] {$f$} (ftX);
   \draw[->] (fX) to node[left] {$p_{f^*X}$} (Y);
 \end{ctikzpicture}
 \item for each $n \in \bbN$, each $X \in \ob_{n+1}\mC$, and each pair of maps $f \colon Y \to \ft X$ and $g \colon Z \to Y$, $(fg)^*X = g^*f^*X$, $1_{\ft X}^* X = X$, $q_{fg, X}= q_{f,X} q_{g, f^*X}$, and $q_{1_{\ft X}, X}= 1_X$.
 \end{enumerate}
\end{definition}

Contextual categories are easily seen to be models for an essentially algebraic theory with sorts indexed by $\bbN + \bbN \times \bbN$. This gives a natural notion of a \textbf{contextual functor}: it is a homomorphism of models of this theory.

The maps $p_X \colon X \to \ft X$ will be called \textbf{basic dependent projections}. A \textbf{dependent projection} is a map which is a composite of basic dependent projections. We will often refer to the objects of a contextual category as contexts. For notational convenience, given $\Gamma \in \ob_n \mC$, we will write $\Gamma.A$ for an object in $\ob_{n+1}\mC$ such that $\ft(\Gamma.A) = \Gamma$. A \textbf{context extension} of a context $\Gamma$ is a context $\Gamma.\Delta$ with a dependent projection $\Gamma.\Delta \to \Gamma$.

Given a contextual category $\mC$, define a contextual category $\mC^\cxt$ as follows (cf.\ \cite[Sec.\ 1.2.3 and 1.3.1]{lumsdaine:thesis}):
\begin{enumerate}
 \item the set of $\ob_n \mC^\cxt$ consists of $n$-iterated context extensions:
 \[ \Gamma_1.\Gamma_2.\ldots .\Gamma_n \]
 in $\mC$ and the morphisms $\Gamma_1.\Gamma_2.\ldots .\Gamma_n \to \Delta_1.\Delta_2.\ldots .\Delta_n$ are morphisms between them regarded as objects of $\mC$;
 \item $\diamond$ is the unique object $\ob_0 \mC^\cxt$;
 \item $\ft(\Gamma_1.\Gamma_2.\ldots .\Gamma_{n+1}) = \Gamma_1.\Gamma_2.\ldots .\Gamma_n$;
 \item the map $p_{\Gamma_1.\Gamma_2.\ldots .\Gamma_{n+1}} \colon \Gamma_1.\Gamma_2.\ldots .\Gamma_{n+1} \to \Gamma_1.\Gamma_2.\ldots .\Gamma_n$ is the dependent projection establishing $\Gamma_1.\Gamma_2.\ldots .\Gamma_{n+1}$ as a context extension;
 \item the chosen pullbacks are given by iterating the pullbacks along basic dependent projections.
\end{enumerate}

There is an obvious contextual functor $\mC \into \mC^\cxt$, which is an equivalence of underlying categories. Moreover, every object of $\mC^\cxt$ is isomorphic to one in $\ob_1 \mC^\cxt$ and every dependent projection is isomorphic to a basic one.

The notion of a contextual category corresponds to the structural rules of type theory. In order to express the logical constructors, we equip a contextual category with some additional structure. Thus for the purpose of this paper, we introduce the following definition:

\begin{definition}
 A \textbf{categorical model of type theory} is a contextual category $\mC$ equipped with $\mathsf{\Pi}$-structure \cite[Def.\ B.1.1]{kapulkin-lumsdaine-voevodsky:simplicial-model}, $\mathsf{\Sigma}$-structure \cite[Def.\ B.1.2]{kapulkin-lumsdaine-voevodsky:simplicial-model}, and $\mathsf{Id}$-structure \cite[Def.\ B.1.3]{kapulkin-lumsdaine-voevodsky:simplicial-model}, and satisfying additionally the $\mathsf{\Pi}$-$\eta$ rule and equipped with Functional Extensionality \cite[Def.\ B.3.1]{kapulkin-lumsdaine-voevodsky:simplicial-model}.
\end{definition}

\begin{remark}
 While we are focusing on the semantic part of the picture, let us observe that on the syntactic side, our choice of type theory is fairly standard. The same selection of rules for those constructors is considered in the HoTT Book \cite{hott:book}, is implemented in the proof assistant \textsc{Coq} 8.3 and later (that is, the distributions used for the UniMath \cite{UniMath} and HoTT \cite{hott:repo} repositories), and has interpretation in various homotopy-theoretic settings, including the simplicial model of Voevodsky \cite{kapulkin-lumsdaine-voevodsky:simplicial-model} and more general presheaf models \cite{shulman:inverse-diagrams}.
\end{remark}

The categorical models of type theory are again models of an essentially algebraic theory with sorts indexed by $\bbN + \bbN \times \bbN$. Let $\Cxl$ denote the category of models for this theory and we will refer to its morphisms as \textbf{homomorphisms of models}.

Let $\mC$ be a categorical model of type theory. By \cite[Sec.\ 1.3.1]{lumsdaine:thesis}, we may lift the structure of a model from $\mC$ to $\mC^\cxt$, and thus given an iterated context extension $\Gamma.\Delta.\Theta$ we may consider e.g.\ $\Gamma.\mathsf{\Pi}(\Delta, \Theta)$. We will heavily exploit this feature, giving definitions for $\mC$ using the structure of $\mC^\cxt$. \footnote{On the syntactic side, this allows us to reason as though we assumed strong $\mathsf{\Sigma}$-types in the theory.} For example, we do so when introducing the notion of a bi-invertible map:

\begin{definition}
Let $\mC$ be a categorical model of type theory. A morphism $f \colon \Gamma \to \Delta$ in $\mC$ is \textbf{bi-invertible} if there exist:
\begin{enumerate}
 \item a morphism $g_1 \colon \Delta \to \Gamma$;
 \item a section $\eta \colon \Gamma \to \langle 1_A, g_1 f \rangle^*\mathsf{Id}_\Gamma$ of the canonical projection $p_{\langle 1_\Gamma, g_1 f \rangle^*\mathsf{Id}_\Gamma} \colon \langle 1_\Gamma, g_1 f \rangle^*\mathsf{Id}_\Gamma \to \Gamma$;
 \item a morphism $g_2 \colon \Delta \to \Gamma$;
 \item a section $\varepsilon \colon \Delta \to \langle 1_\Delta, f g_1 \rangle^*\mathsf{Id}_\Delta$ of the canonical projection $p_{\langle 1_\Delta, f g_1 \rangle^*\mathsf{Id}_\Delta} \colon \langle 1_\Delta, f g_1 \rangle^*\mathsf{Id}_\Delta \to \Delta$.
\end{enumerate}
\end{definition}

In this definition, we consider $\Gamma$ and $\Delta$ as objects in $\mC^\cxt$ in order to speak of $\mathsf{Id}_\Gamma$, $\mathsf{Id}_\Delta$, and so on.

Our definition of a bi-invertible map is an externalization of the notion $\mathsf{isHIso}(f)$ of \cite[Def.\ B.3.3]{kapulkin-lumsdaine-voevodsky:simplicial-model}, which in turn is a translation to the language of contextual categories of the notion of a bi-invertible map in type theory \cite[Def.\ 4.3.1]{hott:book}.

\begin{examples}
The identity map $1_A$ is bi-invertible; the composite of two bi-invertible maps is again bi-invertible.
\end{examples}

\subsection{Abstract homotopy theory}\label{sec:aht}



We will be concerned with two axiomatic approaches to abstract homotopy theory: \emph{categories with weak equivalences} (a.k.a.\ \emph{relative categories} \cite{barwick-kan}) and \emph{quasicategories} \cite{joyal:notes-on-quasicategories},\cite{joyal:theory-of-quasi-cats}, \cite{lurie:htt} (see \cite[Introduction]{szumilo:two-models} for an excellent overview of other approaches). We will therefore review their basic theory and introduce a functor (simplicial localization) associating to a category with weak equivalences a quasicategory. Later on, we will review the relevant material from the theory of quasicategories.

A \textbf{category with weak equivalences} consists of a category $\C$ together with a wide subcategory $\w\C$ (that is, a subcategory containing all objects of $\C$), whose morphisms will be referred to as \textbf{weak equivalences}. A functor between categories with weak equivalences is \textbf{homotopical} if it takes weak equivalences to weak equivalences. The category of categories with weak equivalences and homotopical functors will be denoted by $\weCat$.

There is a natural forgetful functor $\U \colon \Cxl \to \weCat$, taking a contextual category $\mC$ to its underlying category, considered as a category with weak equivalences, in which weak equivalances are bi-invertible maps.

A \textbf{quasicategory} is a simplicial set $\C$ satisfying the inner horn lifting condition, that is, given any map $\Lambda^k[n] \to \C$ with $0 < k < n$, there exists a dashed arrow making the triangle commute:
 \begin{ctikzpicture}
   \matrix[diagram]
   {
     |(a)| \Lambda^k[n] & |(c)| \C \\
     |(b)| \Delta[n] &  \\
   };

   \draw[->] (a) to (c);
   \draw[->,inj] (a) to (b);
   \draw[->,dashed] (b) to (c);
 \end{ctikzpicture}
The full subcategory of $\sSet$ consisting of quasicategories will be denoted $\qCat$.

Barwick and Kan \cite{barwick-kan} constructed a model structure on $\weCat$ and showed that it is Quillen equivalent to the Joyal's model structure on $\sSet$, in which fibrant objects are exactly quasicategories \cite[Thm.\ 6.12]{joyal:theory-of-quasi-cats}. While there are many constructions assigning to a category with weak equivalences a quasicategory that are equivalences of homotopy theories, by result of To\"en \cite[Thm.\ 6.2]{toen:unicity}, all of them are equivalent either to the functor constructed by Barwick and Kan or its opposite (see also \cite[Rmk.\ 4.7]{kapulkin-szumilo}).

To fix the notation, we should choose one such functor; thus, let $\Ho_\infty \colon \weCat \to \sSet$ denote the functor given by the composite of Rezk's classification diagram \cite[Sec.\ 3.3]{rezk:css} followed by the right derived functor of the functor taking a bisimplicial set to its zeroth row \cite[Thm.\ 4.11]{joyal-tierney:qcat-vs-segal}. For the discussion of other possible choices (e.g.\ hammock localization followed by the derived homotopy coherent nerve) see \cite[Sec.\ 1.6]{barwick:k-theory-of-higher-cats}.

In the remainder of this section, we shall review a few notions and facts from quasicategory theory. Our choice of definitions follows \cite{joyal:theory-of-quasi-cats}; another canonical choice is \cite{lurie:htt}.

We will write $\N \colon \Cat \to \sSet$ for the nerve functor and $\tau_1 \colon \sSet \to \Cat$ for its left adjoint. Let $\qCat_2$ denote the $2$-category whose objects are quasicategories and whose hom-categories are given by $\qCat_2(\C,\D) = \tau_1(\D^\C)$. 

Let $E[1]$ denote the nerve of the contractible groupoid on two objects. Two maps $f, g \colon \C \to \D$ between quasicategories are $E[1]$-homotopic (notation: $f \we_{E[1]} g$) if there exists a map $H \colon E[1] \times \C \to \D$ (called $E[1]$-homotopy) such that $H|(\partial\Delta[1] \times \C) = [ f, g ]$ A map $f \colon \C \to \D$ is a \textbf{categorical equivalence} if there exists a map $g \colon \D \to \C$ such that $gf \we_{E[1]} 1_\C$ and $fg \we_{E[1]} 1_\D$.The following lemma gives a useful characterization of categorical equivalences:

\begin{lemma}[{\cite[Lem.\ 4.3]{kapulkin-szumilo}}]\label{lem:cat-equiv-lifting}
 Let $f \colon \C \to \D$ be a map of quasicategories. Suppose that for each $n \in \bbN$ and a square:
 \begin{ctikzpicture}
   \matrix[diagram]
   {
     |(b)| \partial \Delta[n] & |(K)| \C \\
     |(s)| \Delta[n]   & |(L)| \D \\
   };

   \draw[->,inj] (b) to (s);

   \draw[->] (K) to node[right] {$f$} (L);

   \draw[->] (b) to node[above] {$u$} (K);
   \draw[->] (s) to node[above] {$v$} (L);
 \end{ctikzpicture}
there are: a map $w \colon \Delta[n] \to \C$ such that $w|\partial \Delta[n] = u$ and an $E[1]$-homotopy from $fw$ to $v$ relative to the boundary. Then $f$ is a categorical equivalence.
\end{lemma}

The \textbf{join} of simplicial sets is a unique functor $\star \colon \sSet \times \sSet \to \sSet$ with natural transformations $K \to K \star L \leftarrow L$ such that $\Delta[m] \star \Delta[n] \iso \Delta[m+1+n]$, naturally in both $m$ and $n$, and both functors $- \star L \colon \sSet \to L \slice \sSet$ and $K \star - \colon \sSet \to K \slice \sSet$ preserve colimits.

Explicitly, the $n$-simplices of $K \star L$ are given by:
\[ (K \star L)_n = \coprod_{\overset{i,j \geq -1}{i+j = n-1}} K_i \times L_j  \]
where we assume that $K_{-1} = L_{-1} = \{* \}$, and the canonical morphism $L \to K \star L$ is taking $y \in L$ to $ (*, y) \in K \star L$. By the Adjoint Functor Theorem, $- \star L$ admits a right adjoint taking a simplicial morphism $X \colon L \to K$ to the simplicial set $K \slice X$ defined as:
 \[ (K \downarrow X)_n = \Big\{ Y \colon \Delta[n] \star L \to K ~\Big| ~ Y|L = X \Big\}. \]

 Instantiating it with the inclusion of a $0$-simplex $x \colon \Delta[0] \to \C$, we obtain the notion of a \textbf{slice quasicategory}
  \[ (\C \slice x)_n = \Big\{ X \colon \Delta[n+1] \to \C ~\Big| ~ X|\Delta^{\{n+1\}} = x \Big\}.\]
  
The canonical map $\C \slice x \to \C$ is an inner fibration and hence $\C \slice x$ is again a quasicategory. Dually, one may define coslice quasicategories $x \slice \C$. Finally, given a simplicial map $G \colon \D \to \C$ between quasicategories and an object $x \in \C$, one defines the comma quasicategory $x \slice G$ as the pullback:
 \begin{ctikzpicture}
   \matrix[diagram]
   {
     |(a)| x \slice G & |(b)| \D \\
     |(c)| x \slice \C   & |(d)| \C \\
   };
   \draw[->] (a) to (b);
   \draw[->] (a) to (c);
   \draw[->] (b) to node[right] {$G$} (d);
   \draw[->] (c) to (d);
 \end{ctikzpicture}
 
The map $x \slice G \to \D$ is again a fibration (as a pullback of a fibration) and hence $x \downarrow G$ is a quasicategory.

Next, we shall define limits in quasicategories. For notational convenience, we will write $K^\triangleleft$ for $\Delta[0] \star K$.
  
\begin{definition}\leavevmode \label{def:limit}
\begin{itemize}
 \item Let $\C$ be a quasicategory and $X \colon K \to \C$ a map of simplicial sets. A \textbf{cone} over $X$ is a simplicial map $Y \colon K^\triangleleft \to \C$ such that $Y|K = X$.
 \item A cone $\tX \colon K^\triangleleft \to \C$ is \textbf{universal} (or a \textbf{limit}) if for all $n > 0$ and all $Z \colon \partial \Delta[n] \star K \to \C$ such that $Z| \Delta^{\{n\}} \star K = \tX$, there exists an extension: 
 \begin{ctikzpicture}
   \matrix[diagram]
   {
     |(a)| \partial \Delta[n] \star K & |(c)| \C \\
     |(b)| \Delta[n] \star K &  \\
   };

   \draw[->] (a) to node[above] {$Z$} (c);
   \draw[->,inj] (a) to (b);
   \draw[->,dashed] (b) to (c);
 \end{ctikzpicture}
 \end{itemize}
\end{definition}

\begin{examples}\leavevmode
\begin{itemize}
 \item A \textbf{terminal object} in a quasicategory $\C$ is a limit of $\varnothing \to \C$.
 \item A \textbf{pullback} in a quasicategory $\C$ is a limit of a diagram $\Lambda^2[2] \to \C$.
\end{itemize}
\end{examples}

Of course, all of the above notions readily dualize to yield coslice quasicategories, colimits, etcetera.

We say that a quasicategory $\C$ has limits of shape $K$ if for every diagram $X \colon K \to \C$, there is a universal cone $\tX \colon K^\triangleleft \to \C$ over $X$. We say that $\C$ has pullbacks if it has limits of shape $K$ for $K = \Lambda^2[2]$.

Given a quasicategory $\C$ with limits of shape $K$, let $\C^{K^{\triangleleft}}_{\mathrm{univ}}$ denote the full simplicial subset of $\C^{K^\triangleleft}$ spanned by those cones over $K$ that are universal. By \cite[Prop.\ 4.3.2.15]{lurie:htt}, the restriction map $\C^{K^\triangleleft}_{\mathrm{univ}} \to \C^K$ is an acyclic fibration and thus admits a section---this defines the limit functor $\lim_K \colon \C^K \to \C^{K^\triangleleft}_{\mathrm{univ}}$.

We conclude this section with the definition and a characterization of adjoints between quasicategories.

\begin{definition}
 Let $F \colon \C \rightleftarrows \D : \! G$ be simplicial maps between quasicategories and let $\eta \colon \C \times \Delta[1] \to \C$ be such that $\eta|\C \times \partial \Delta[1] = [1_\C, GF]$ and $\varepsilon \colon \D \times \Delta[1] \to \D$ such that $\varepsilon|\D \times \partial\Delta[1] = [FG, 1_\D]$. The quadruple $(F,G,\eta,\varepsilon)$ is an \textbf{adjunction of quasicategories} (with $F$ the \textbf{left adjoint} and $G$ the \textbf{right adjoint}) if $(F,G,\eta,\varepsilon)$ is an adjunction in the $2$-category $\qCat_2$.
 \end{definition}

\begin{lemma}[{\cite[Sec.\ 17.4]{joyal:notes-on-quasicategories}}] \label{lem:adjoint-characterization}
A simplicial map $G \colon \D \to \C$ between quasicategories is a right adjoint if and only if for any $x \in \C$, the comma quasicategory $x \slice G$ has an initial object.
\end{lemma}

\subsection{Statement of the main theorem}\label{sec:statement}

In this section, we state our main theorem and explain the proof strategy. To do so, we need the notion of a locally cartesian closed quasicategory and, in turn, the pullback functor. In the previous section, given a quasicategory $\C$ with pullbacks, we defined a functor $\C^{\Lambda^2[2]} \to \C^{\Delta[1] \times \Delta[1]}_{\mathrm{univ}}$ taking a cospan in $\C$ to its pullback square. However, we would also like to consider pullback as a functor between slices of quasicategories.

To this end, given a $1$-simplex $f \colon x \to y$ in a quasicategory $\C$ with pullbacks, we shall construct a functor $f^* \colon \C \slice y \to \C \slice x$, taking a $1$-simplex to its pullback along $f$. Let $\C^{(\Delta[1] \times \Delta[1])_f}_{\mathrm{univ}}$ denote the full simplicial subset of $\C^{\Delta[1] \times \Delta[1]}_{\mathrm{univ}}$ spanned by those universal cones $\Delta[1] \times \Delta[1] \to \C$ that send the edge $\Delta^{\{1\}} \times \Delta[1]$ to $f \in \C$. We therefore obtain a pullback square:
 \begin{ctikzpicture}
   \matrix[diagram]
   {
     |(a)| \C^{(\Delta[1] \times \Delta[1])_f}_{\mathrm{univ}} 
                                                                & |(b)| \C^{\Delta[1] \times \Delta[1]}_{\mathrm{univ}} \\
     |(c)| \C \slice y   & |(d)| \C^{\Lambda^2[2]} \\
   };
   \draw[->] (a) to (b);
   \draw[->>] (a) to node[left] {$\we$} (c);
   \draw[->>] (b) to node[right] {$\we$} (d);
   \draw[->] (c) to (d);
 \end{ctikzpicture}
Since the left vertical map is an acyclic fibration (as a pullback of one), it admits a section. Postcomposing this section with the projection $\C^{(\Delta[1] \times \Delta[1])_f}_{\mathrm{univ}} \to \C \slice x$, we obtain the pullback functor $f^* \colon \C \slice y \to \C \slice x$.

\begin{definition}
 A quasicategory $\C$ is \textbf{locally cartesian closed} if it has finite limits and for any $1$-simplex $f \colon x \to y$, the pullback functor $f^* \colon \C \slice y \to \C \slice x$ has a right adjoint.
\end{definition}

We are now ready to state our main theorem:

\begin{theorem} \label{thm:main-first-statement}
 For any contextual category $\mC$, the quasicategory $\Ho_{\infty}\U \mC$ is locally cartesian closed.
\end{theorem}

This theorem will be proven as the last theorem in the paper (\Cref{thm:main-second-statement}). Let us also note that this gives a partial solution of the conjecture posed by Andr\'e Joyal during the \emph{Oberwolfach Mini-Workshop 1109a: Homotopy Interpretation of Constructive Type Theory} \cite[Conj.\ 1]{joyal:remarks-on-homotopical-logic} since the simplicial localization functor $\L$ appearing in his statement is equivalent to $\Ho_{\infty}$ defined above (see \cite[Ex.\ 1.6.3]{barwick:k-theory-of-higher-cats}).

The main difficulty in proving \Cref{thm:main-first-statement} lies in the fact that simplicial localization functors $\weCat \to \sSet$ involve inexplicit fibrant replacements that makes the resulting quasicategories hard to work with. Our proof strategy is based on an observation \cite[Thm.\ 3.2.5]{avigad-kapulkin-lumsdaine} that every categorical model of type theory carries the structure of a \emph{fibration category}. On the other hand, given a fibration category, one may associate to it its quasicategory of frames, defined by Szumi{\l}o \cite{szumilo:two-models}. This construction does not involve any fibrant replacements (thus allowing one to work directly with the category at hand) and, by the results of \cite{kapulkin-szumilo}, the quasicategory of frames in a fibration category is equivalent to its simplicial localization.

Our main goal is to identify additional conditions on a fibration category that guarantee that its quasicategory of frames is locally cartesian closed. This is possible essentially because the construction of the quasicategory of frames preserves adjoint pairs of exact functors and is compatible with passing to slice categories. Finally, we verify that for a categorical model of type theory $\mC$, the underlying fibration category of $\mC$ satisfies these additional conditions.

\section{Fibration categories and the quasicategory of frames} \label{sec:fib-cat}

In this section, we review the definition and basic properties of fibration categories, recall the construction of the quasicategory of frames, and gather a few important lemmas.

\begin{definition}[Brown {\cite{brown:abstract-homotopy-theory}}] \footnote{Our formulation is more restrictive than Brown's original axiomatization in that we require the class of weak equivalences to satisfy the 2-out-of-6 condition, whereas he required only 2-out-of-3. In the presence of the other axioms, F1 is then equivalent to 2-out-of-3 together with the fact that the class of weak equivalences is saturated, i.e.\ only weak equivalences are inverted in $\Ho \C$ (see \cite[Thm.\ 7.2.7]{radulescu-banu})} \label{def:fibration-category}
 A \textbf{fibration category} consists of a category $\C$ together with two wide subcategories: of \textbf{fibrations} (whose morphisms will be denoted $\fib$) and \textbf{weak equivalences} (whose morphisms we denote $\overset\we\to$) such that (in what follows, an \textbf{acyclic fibration} is a map that is both a fibration and a weak equivalence):
 \begin{enumerate}
  \item[F1.] weak equivalences satisfy the \textbf{2-out-of-6} property; that is, given a composable triple of morphisms:
  \[ X \overset{f}\longrightarrow Y \overset{g}\longrightarrow Z \overset{h}\longrightarrow Z \]
  if both $hg$ and $gf$ are weak equivalences, then $f$, $g$, and $h$ are all weak equivalences.
  \item[F2.] all isomorphisms are acyclic fibrations.
  \item[F3.] pullbacks along fibrations exist; fibrations and acyclic fibrations are stable under pullback.
  \item[F4.] $\C$ has a terminal object $1$; the canonical map $X \to 1$ is a fibration for any object $X \in \C$ (that is, all objects are \textbf{fibrant}).
  \item[F5.] every map can be factored as a weak equivalence followed by a fibration.
 \end{enumerate}
\end{definition}


\begin{examples} \leavevmode
\begin{enumerate}
 \item For a model category $\M$, its full subcategory of fibrant objects carries a fibration category structure.
 \item Let $\C$ be a fibration category and $A \in \C$. Denote by $\C \fibslice A$ be the full subcategory of the slice $\C \downarrow A$ consisting only of fibrations $B \fib A$, and declare a map in $\C \fibslice A$ to be a fibration (resp.\ a weak equivalence) if its image under the canonical functor $\C \fibslice A \to \C$ is a fibration (resp.\ a weak equivalence) in $\C$. This gives $\C \fibslice A$ the structure of a fibration category. 
\end{enumerate}
\end{examples}

The following theorem justifies (at least partially) our interest in fibration categories:

\begin{theorem}[{\cite[Thm.\ 3.2.5]{avigad-kapulkin-lumsdaine}}] \label{thm:akl}
Every categorical model of type theory $\mC$ carries the structure of a fibration category, in which fibrations are maps isomorphic to a composite of dependent projections and weak equivalences are the bi-invertible maps.
\end{theorem}

(The construction in \cite{avigad-kapulkin-lumsdaine} uses the assumption that every dependent projection is isomorphic to a basic one. This always holds in $\mC^\cxt$; so we may apply the construction there, and then transfer the fibration category structure back to $\mC$ along the canonical equivalence $\mC \to \mC^\cxt$.)

 A \textbf{path object} for $A$ in a fibration category $\C$ is any factorization of the diagonal map $A \to A \times A$ into a weak equivalence followed by a fibration $A \overset{\sim}{\to} PA \fib A \times A$. Let $f,g \colon A \to B$ be a pair of maps in $\C$. A \textbf{right homotopy} between $f$ and $g$ is a commutative square:
 \begin{ctikzpicture}
   \matrix[diagram]
   {
     |(a)| A' & |(b)| PB \\
     |(c)| A & |(d)| B \times B \\
   };

   \draw[->] (a) to (b);
   \draw[->>] (a) to node[left] {$\we$} (c);
   \draw[->] (b) to (d);
   \draw[->] (c) to node[above] {$\langle f , g \rangle$} (d);
 \end{ctikzpicture}
 We say that $f$ and $g$ are \textbf{right-homotopic} if there is a right homotopy between $f$ and $g$.

\begin{theorem}\leavevmode \label{thm:right-homotopy-gives-homotopy}
 \begin{enumerate}
  \item The relation of being right-homotopic is an equivalence relation on $\Hom(A,B)$ and is respected by pre- and postcomposition \cite[Thm.\ 6.3.3.(2)]{radulescu-banu}.
  \item Two morphisms $f,g \colon A \to B$ are right-homotopic if and only if they are homotopic (i.e.~equal in $\Ho\C$) \cite[Thm.\ 6.3.1.(2)]{radulescu-banu}.
 \end{enumerate}
\end{theorem}

\begin{definition}\leavevmode
\begin{enumerate}
\item A functor between fibration categories is \textbf{exact} if it preserves fibrations, acyclic fibrations, pullbacks along fibrations, and the terminal object.
\item An exact functor is a \textbf{weak equivalence} of fibration categories if it induces an equivalence of homotopy categories.
\end{enumerate}
\end{definition}

\begin{examples}\leavevmode
\begin{enumerate}
 \item Every right Quillen functor induces an exact functor between subcategories of fibrant objects.
 \item Suppose an exact functor $F \colon \C \to \D$ is a homotopy equivalence, i.e.\ there exists a homotopical functor $G \colon \D \to \C$ (not necessarily exact) together with two natural weak equivalences $GF \overset\we\to 1_\C$ and $1_\D \overset\we\to FG$. Then $F$ is a weak equivalence of fibration categories.
 \item Let $F \colon \mC \to \mathsf{D}$ be a homomorphism of categorical models of type theory. Then $F$ is an exact functor between the fibration categories of \Cref{thm:akl} associated to $\mC$ and $\mathsf{D}$. Indeed, $F$ has to preserve dependent projections (and hence fibrations), distinguished pullbacks (and hence pullbacks along fibrations), and the terminal object $\diamond$; moreover, one easily sees that $F$ preserves acyclic fibrations as well.
 \item Given a morphism $f \colon A \to B$ in a fibration category $\C$, the pullback functor $f^* \colon \C \fibslice B \to \C \fibslice A$ is an exact functor.
 \item Let $f \colon A \to B$ be a fibration in $\C$. The functor $f_! \colon \C \fibslice A \to \C \fibslice B$ given by $f_!p = fp$ is exact if and only if $f$ is an isomorphism (otherwise, $f_!$ does not preserve the terminal object).
\end{enumerate} 
\end{examples}

\begin{proposition}\label{prop:pb-along-we}
 Let $\C$ be a fibration category and $f \colon A \to B$ an acyclic fibration. Then the pullback functor $f^* \colon  \C \fibslice B \to \C \fibslice A$ is a weak equivalence.
\end{proposition}

\begin{proof}
 It follows from the fact that $f^*$ is a homotopy equivalence, with the homotopy inverse given by $f_!$ since the unit and counit of the adjunction $f_! \adjoint f^*$ are natural weak equivalences.
\end{proof}

\begin{definition}\label{def:direct-inverse-cat}\leavevmode
 \begin{enumerate}
  \item A category $J$ is \textbf{inverse} if there is a function $\deg \colon \ob J \to \bbN$ such that for every non-identity map $j \to j'$ in $J$ we have $\deg j > \deg j'$.
 \end{enumerate}
Let $J$ be an inverse category.
\begin{enumerate}
 \item[(2)] Let $j \in J$. The \textbf{matching category} $\partial(j \slice J)$ of $j$ is the full subcategory of the slice category $j \slice J$ consisting of all objects except $1_j$. There is a canonical functor $\cod \colon \partial(j\slice J) \to J$, assigning to a morphism (regarded as object of $\partial(j \slice J)$) its codomain.
 \item[(3)] Let $X \colon J \to \C$ and $j \in J$. The \textbf{matching object} of $X$ at $j$ is defined as a limit of the composite
   \[ M_j X := \lim (\partial (j \slice J) \longrightarrow J \overset{X}{\longrightarrow} \C),\]
   where the first map sends an arrow in the slice to its codomain and the second is $X$. The canonical map $X_j \to M_j X$ is called the \textbf{matching map}.
 \item[(4)] Let $\C$ be a fibration category. A diagram $X \colon J \to \C$ is called \textbf{Reedy fibrant} if for all $j \in J$, the matching object $M_j X$ exists and the matching map $X_j \to M_j X$ is a fibration.
\end{enumerate}
\end{definition}
 
Recall that a \textbf{homotopical category} is a category with weak equivalences $\C$ satisfying the 2-out-of-6 property. The full subcategory of $\weCat$ consisting of homotopical categories will be denoted $\hoCat$. From this point on, all categories that we will consider will be homotopical (sometimes with trivial homotopical structure given by isomorphisms only, e.g.\ $[n]$) and all functors are assumed to be homotopical.

Given a homotopical category $J$, we will construct a homotopical category $DJ$ such that $DJ^\op$ is inverse, together with a homotopical functor $\ev_0 \colon DJ^\op \to J$. The objects of $DJ$ are pairs $([n], \varphi \colon [n] \to J)$, where $n \in \bbN$ and $\varphi$ is an arbitrary functor. A map
  \[ f \colon ([n], \varphi) \to ([m], \psi)  \]
 is an injective, order preserving map $f \colon [n] \into [m]$ such that $\psi f = \varphi$. For notational convenience later on, we will write $i^k$ for the constant function $[k] \to [n]$ taking value $i$.
 
 Clearly, $DJ^\op$ is an inverse category (with $\deg ([n], \varphi) = n$). Putting $\ev_0([n],\varphi)=\varphi(0)$ gives a contravariant functor and the homotopical structure on $DJ$ is created by this functor.

Let us also mention some variations on this construction. Given a homotopical poset $P$, we define a homotopical category $\Sd P$ (the \textbf{barycentric subdivision} of $P$) as the full subcategory of $DP$ consisting of functors $[k] \to P$ that are injective on objects. Often when describing objects of $\Sd P$, we will identify an injective map $\varphi \colon [k] \to P$ with its image $\mathrm{im}\varphi \subseteq P$.

We may also define $DK$ for a simplicial set $K$ by the left Kan extension of the functor $\Delta \into \hoCat \overset{D}\longrightarrow \hoCat$ defined above along the Yoneda embedding. Explicitly, the underlying category of $DK$ is the full subcategory of the category of elements of $K \colon \Delta^\op \to \Set$ spanned by the face operators and the set of weak equivalences is the smallest set closed under 2-out-of-6 and containing the morphisms induced by the degenerate $1$-simplices of $K$.

Following \cite[Def.\ 3.17]{szumilo:two-models}, a \textbf{marked simplicial complex} is a simplicial set $K$ with a chosen inclusion $K \into \N P$, where $P$ is a homotopical poset. Given a marked simplicial complex $K$ define $\Sd K$ as the full homotopical subcategory of $DK$ whose objects are nondegenerate simplices of $K$.

Finally, we will also define an \emph{augmented} version of $D$, denoted by $D_a$. For a homotopical category $J$, define the homotopical category $D_a J$ as follows: the underlying category of $D_a J$ is obtained by freely adjoining an initial object $\varnothing$ to $DJ$ and the weak equivalences are those of $DJ$ (none of the newly added maps is a weak equivalence).

  Let $\C$ be a fibration category. Define a simplicial set $\Nf\C$, called the \textbf{quasicategory of frames} of $\C$, by:
 \[ (\Nf\C)_m := \left\{\text{homotopical, Reedy fibrant diagrams } D[m]^\op \to \C \right\}. \]
 The assignment $\C \mapsto \Nf\C$ extends to a functor on the category of fibration categories and exact functors. In fact, $\Nf\C$ can be characterized as a representing object for a certain functor $\sSet \to \Set$:

\begin{lemma}[{\cite[Prop.~3.7]{szumilo:two-models}}] \label{thm:fake-adjunction}
 Let $\C$ be a fibration category and $K$ a simplicial set. Then there is a natural bijection between simplicial maps $K \to \Nf\C$ and homotopical, Reedy fibrant functors $DK^\op \to \C$. In other words, the functor taking a simplicial set $K$ to the set of homotopical, Reedy fibrant diagrams $DK^\op \to \C$ is representable, represented by $\Nf\C$.
\end{lemma}
 
 The following two theorems indicate the importance of this construction for our purposes:

 \begin{theorem}[Szumi{\l}o] \label{thm:szumilo-main} \leavevmode
 \begin{enumerate}
  \item The functor $\Nf$ takes values in the category of finitely complete quasicategories and finite limit preserving functors \cite[Thm.~3.3]{szumilo:two-models}.
  \item Moreover, it takes weak equivalences of fibration categories to categorical equivalences of quasicategories \cite[Thm.~3.26]{szumilo:two-models}.
 \end{enumerate}
\end{theorem}

\begin{theorem}[{\cite[Cor.\ 4.6]{kapulkin-szumilo}}]\label{thm:Ho-Nf-coincide}
 For any fibration category $\C$, the quasicategories $\Ho_\infty\C$ and $\Nf\C$ are weakly equivalent in Joyal's model structure.
\end{theorem}

It thus follows formally from the previous two theorems together with \Cref{thm:akl} that for every contextual category $\mC$, the quasicategory $\Ho_\infty\U\mC$ has finite limits. We still need to show, of course, that $\Ho_\infty\U\mC$ has right adjoints to pullback functors. However, there is a very convenient characterization of pullbacks in the quasicategories $\Nf\C$:

\begin{example}[{\cite[Ex.\ 3.34]{szumilo:two-models}}] \label{ex:pullback-in-NfC}\leavevmode
A square $P \colon D([1] \times [1])^\op \to \C$ in $\Nf\C$ is a homotopy pullback if in its restriction to $\Sd([1] \times [1])^\op$
\begin{center}
 \begin{ctikzpicture}
    \matrix[diagram,column sep=10em,row sep=10em]
    {
      |(a-a)| P_{0,0} & |(b-a)| P_{1,0} \\
      |(a-b)| P_{0,1} & |(b-b)| P_{1,1} \\
    };

    \node (aa-ab) at (barycentric cs:a-a=1,a-b=1,b-b=0) {$P_{00,01}$};
    \node (ab-bb) at (barycentric cs:a-a=0,a-b=1,b-b=1) {$P_{01,11}$};
    \node (ab-ab) at (barycentric cs:a-a=1,a-b=0,b-b=1) {$P_{01,01}$};

    \node (ab-aa) at (barycentric cs:a-a=1,b-a=1,b-b=0) {$P_{01,00}$};
    \node (bb-ab) at (barycentric cs:a-a=0,b-a=1,b-b=1) {$P_{11,01}$};

    \node (aab-abb) at (barycentric cs:a-a=2,a-b=3,b-b=2) {$P_{001,011}$};
    \node (abb-aab) at (barycentric cs:a-a=2,b-a=3,b-b=2) {$P_{011,001}$};

    \draw[->]  (aa-ab) to node[left] {$\we$} (a-a);
    \draw[->] (aa-ab) to (a-b);

    \draw[->] (ab-bb) to node[below] {$\we$}  (a-b);
    \draw[->] (ab-bb) to (b-b);

    \draw[->] (ab-aa) to node[above] {$\we$} (a-a);
    \draw[->] (ab-aa) to (b-a);

    \draw[->] (bb-ab) to node[right] {$\we$} (b-a);
    \draw[->] (bb-ab) to (b-b);

    \draw[->] (ab-ab) to node[above right] {$\we$} (a-a);
    \draw[->,dashed] (ab-ab) to (b-b);

    \draw[->] (aab-abb) to node[above right] {$\we$} (aa-ab);
    \draw[->,dashed] (aab-abb) to (ab-bb);
    \draw[->] (aab-abb) to node[below right]  {$\we$} (ab-ab);

    \draw[->] (abb-aab) to node[below left] {$\we$} (ab-aa);
    \draw[->,dashed] (abb-aab) to  (bb-ab);
    \draw[->] (abb-aab) to node[below right] {$\we$} (ab-ab);
  \end{ctikzpicture}
 \end{center} 
 the canonical map $P_{001,011} \times_{P_{01,01}} P_{011,001} \to P_{01,11} \times_{P_{1,1}} P_{11,01}$ induced by the dashed arrows is a weak equivalence.
\end{example}

The next three lemmas from \cite{szumilo:two-models} will be used throughout the paper:

\begin{lemma}[{\cite[Lem.~3.19 and Lem.\ 1.24 (1)$\Rightarrow$(2)]{szumilo:two-models}}]\label{thm:Sd-D-acyclic}
 Let $K \into L $ be an inclusion of finite marked simplicial complexes and let $\C$ be a fibration category. Given a Reedy fibrant diagram $X \colon (DK \cup \Sd L)^\op \to \C$, there exists a Reedy fibrant diagram $\tX \colon DL^\op \to \C$ such that $\tX|(DK \cup \Sd L)^\op = X$.
\end{lemma}

(By an inclusion of marked simplicial complexes we understand an injective map that reflects equivalences.)

\begin{lemma}[{\cite[Lem.\ 4.6]{szumilo:two-models}}] \label{lem:rel-homot}
Let $\C$ be a fibration category and $K \into L$ and inclusion of marked simplicial complexes.
Moreover, let $X, Y \colon DL^\op \to \C$ be homotopical, Reedy fibrant diagrams and $f \colon X| (\Sd L)^\op \to Y| (\Sd L)^\op$ a natural weak equivalence such that $f|(\Sd K)^\op$ is the identity transformation.
Then $X$ and $Y$ are $E[1]$-homotopic relative to $K$ as diagrams in $\Nf \C$.
\end{lemma}

Recall that a \textbf{cosieve} is functor $I \into J$ that is injective on objects, fully faithful, and such that if $i \in I$ and there is $i \to j$, then $j \in I$.

\begin{lemma}[{\cite[Prop.\ 1.19.(1)]{szumilo:two-models}}]\label{lem:extend-reedy}
 Let $I \into J$ be a cosieve of inverse categories and $\C$ a fibration category. Moreover, let $X \colon I \to \C$ be a Reedy fibrant diagram and $Y \colon J \to \C$ any diagram with a natural weak equivalence $w \colon Y|I \overset\we\to X$. Then there exists a Reedy fibrant diagram $\tX \colon J \to \C$ such that $\tX|I = X$, together with $\widetilde{w} \colon Y \to \tX$ such that $\widetilde{w}|I = w$.
\end{lemma}

Let $F \colon \C \to \D$ be a homotopical (but not necessarily exact) functor between fibration categories. We define a simplicial set $\C_F$ by:
 \[ (\C_F)_m := \{ (X \in (\Nf\C)_m, Y \in (\Nf\D)_m, w \colon FX \overset{\sim}{\to} Y) \} \]
 Intuitively, an $m$-simplex of $\C_F$ consists of an $m$-simplex of $X \in (\Nf\C)_m$, together with a fibrant replacement $Y$ of $FX$.
 
 \begin{lemma}\label{lem:CF-to-NfC-acyclic-fibration}
 For any homotopical functor $F \colon \C \to \D$ between fibration categories, the canonical projection $\C_F \to \Nf\C$ is an acyclic fibration. In particular,  $\C_F$ is a quasicategory.
\end{lemma}

\begin{proof}
We need to find a lift for the following square:
 \begin{ctikzpicture}
   \matrix[diagram]
   {
     |(a)| \partial \Delta[m]  & |(b)| \C_F \\
     |(c)| \Delta[m] & |(d)| \Nf\C \\
   };

   \draw[->] (a) to (b);
   \draw[->,inj] (a) to (c);
   \draw[->] (b) to (d);
   \draw[->] (c) to (d);
 \end{ctikzpicture}
Thus we are given:
\begin{itemize}
 \item homotopical, Reedy fibrant $X \colon D[m]^\op \to \C$ (bottom map)
 \item homotopical, Reedy fibrant $Y \colon D(\partial \Delta[m])^\op \to \D$, together with a natural weak equivalence $w \colon FX|D(\partial \Delta[m])^\op \to Y$ (top map).
\end{itemize}
We need to find extensions of $Y$ and $w$ to $D[m]^\op$. This follows immediately by \Cref{thm:fake-adjunction,lem:extend-reedy}.
\end{proof}

\begin{construction}\label{rmk:Nf-on-homotopical}
This allows us to define the value of $\Nf$ on a functor between fibration categories that is not necessarily exact. Indeed, by \Cref{lem:CF-to-NfC-acyclic-fibration}, the canonical projection $\C_F \to \Nf\C$ is an acyclic fibration, and hence admits a section. Postcomposing this section with the projection $\C_F \to \Nf\D$ yields a map $\Nf\C \to \Nf\D$ that we may take as $\Nf F$.

Suppose $F \colon \C \rightleftarrows \D : \! G$ are adjoint functors between fibration categories with $G$ exact and $F$ homotopical. Assume moreover that both the unit and the counit of this adjunction are natural weak equivalences. Then the diagram:
 \begin{ctikzpicture}
   \matrix[diagram]
   {
                     & |(a)| \C_F &                \\
     |(b)| \Nf\D &                & |(c)| \Nf\C \\
   };

   \draw[->] (a) to (b);
   \draw[->] (a) to (c);
   \draw[->] (b) to node[below] {$\Nf G$} (c);
 \end{ctikzpicture}
 commutes up to $E[1]$-homotopy by \Cref{lem:rel-homot}. Since by \Cref{thm:szumilo-main}, $\Nf G$ is known to be a categorical equivalence, it follows that $\Nf F$ and $\Nf G$ are each other homotopy inverses.

In particular, given a fibration $f \colon A \to B$ in a fibration category $\C$, one gets a homotopical functor $f_! \colon \C \fibslice A \to \C \fibslice B$ and thus we may define $\Nf f_!$ as above. Unwinding the definitions, we see that $\Nf f_! X$ is given by applying the functor $f_!$ to $X$ pointwise and then replacing the resulting diagram fibrantly. If $f$ was an acyclic fibration to start with, the functors $\Nf f_!$ and $\Nf f^*$ are clearly seen to be homotopy inverses in Joyal's model structure on $\sSet$. \qed
\end{construction}

\section{Slices of the quasicategory of frames} \label{sec:slices}

Let $\C$ be a fibration category and $A \colon D[0]^\op \to \C$ a $0$-simplex in $\Nf\C$. Our main goal in this section is to construct a map $\Nf\C\slice A \to \Nf(\C \fibslice A_0)$ and prove that it is an equivalence. We begin by describing the $n$-simplices in each of these quasicategories.

An $n$-simplex in $\Nf\C \slice A$ is a homotopical, Reedy fibrant diagram $X\colon D[n+1]^\op \to \C$ such that $X_{(n+1)^k} = A_{0^k}$.

An $n$-simplex in $\Nf(\C \fibslice A_0)$ is a homotopical, Reedy fibrant diagram $D[n]^\op \to \C \fibslice A_0$ or equivalently, it is a homotopical, Reedy fibrant diagram $\tA\colon D_a [n]^\op \to \C$.

There is a natural inclusion $D_a[n] \into D[n+1]$ sending $\varphi \colon [k] \to [n] \in D_a[n]$ to $\varphi' \colon [k+1] \to [n+1]$ defined by:
\[ \varphi'(i) = \left\{ \begin{array}{ll}
\varphi(i) & \text{ if } i \leq k \\
n+1        & \text{ if } i = k+1
\end{array}\right. \]
That is, the objects of $D_a[n]$, regarded as a subcategory of $D[n+1]$ are:
\[ \{ [k] \overset{\varphi}{\to} [n+1]~|~\varphi(k)=n+1 \text{ and } \varphi(k-1) \leq n\}.\]

Given $\varphi \colon [k] \to [n] \in D_a[n]$, the induced map
\[ D_a[n] \slice \varphi \to D[n+1] \slice \varphi' \]
is easily seen to be a sieve. We claim that it is moreover cofinal. By \cite[Thm.~IX.3.1]{mac-lane:cwm}, it suffices to show that that for any $\psi \colon [l] \to [n+1] \in D[n+1]$, the slice category $\psi \slice (D_a[n] \slice \varphi')$ is nonempty and connected. If $l = k$, this is clear; otherwise, we must have $l < k$, and thus $\psi' \colon [l+1] \to [n+1]$ (where $\psi \mapsto \psi'$ is as defined above) is the initial object $\psi \slice (D_a[n] \slice \varphi')$.

Thus the inclusion $D_a[n] \into D[n+1]$ satisfies the assumptions of \cite[Lem.\ 2.16]{kapulkin-szumilo} and precomposing with it gives a well-defined map  $\S_A \colon \Nf\C \slice A \to \Nf(\C \fibslice A_0)$. Explicitly, given an $n$-simplex $X\colon D[n+1]^\op \to \C$ in $\Nf\C \slice A$, the $n$-simplex $\S_A X \colon D_a[n]^\op \to \C$ in $\Nf(\C \fibslice A_0)$ is given by:
\[ (\S_A X)_{a_1 a_2 \ldots a_i} := X_{a_1 a_2 \ldots a_i (n+1)}. \]

\begin{theorem}\label{thm:slices}
 Let $\C$ be a fibration category and let $A \colon D[0]^\op \to \C$ be a $0$-simplex in $\Nf\C$. Then the map $\S_A \colon \Nf\C \slice A \to \Nf(\C \fibslice A_0)$ is a categorical equivalence.
\end{theorem}

\begin{proof}
 By \Cref{lem:cat-equiv-lifting}, we need to find a filler $\tX$ for the following square:
 \begin{ctikzpicture}
   \matrix[diagram]
   {
     |(a)| \partial \Delta [n] & |(b)| \Nf\C \slice A \\
     |(c)|  \Delta [n] & |(d)| \Nf(\C \fibslice A_0) \\
   };

   \draw[->] (a) to node[above] {$X$} (b);
   \draw[->,inj] (a) to (c);
   \draw[->] (b) to node[right] {$\S_A$} (d);
   \draw[->] (c) to node[above] {$X'$} (d);
 \end{ctikzpicture}
By the (join $\adjoint$ slice)-adjunction, a simplicial map $X \colon \partial \Delta[n] \to \Nf\C \slice A$ corresponds naturally to a map $\partial \Delta [n] \star \Delta [0] = \Lambda^{n+1}[n+1] \to \Nf(\C)$ whose restriction to the $(n+1)$-st vertex is $A$. By \Cref{thm:fake-adjunction}, this in turn corresponds naturally to a diagram $X \colon D(\Lambda^{n+1}[n+1])^{\op} \to \C$ extending $A$, as explained in the following paragraph.

Notice that $D(\Lambda^{n+1}[n+1])$ is the subcategory of $D[n+1]$ consisting of those monotone functions $\varphi \colon [k] \to [n+1]$ that skip some $i \in [n] \subseteq [n+1]$. That is, the objects of $D(\Lambda^{n+1}[n+1]) $ are:
\[ \{ [k] \overset{\varphi}{\to} [n+1]~|~\text{there exists } i \leq n \text{ s.th.~} i \not\in \mathrm{im}(\varphi) \} \]
and thus, in particular, we have $X_{(n+1)^k} = A_{0^k}$.

Moreover, an $n$-simplex $X' \colon \Delta[n] \to \Nf(\C \fibslice A_0)$ corresponds naturally to a diagram $X' \colon D_a[n]^\op \to \C$ with $X'_{\varnothing} = A_0$.

We seek an extension $\tX \colon D[n+1]^\op \to \C$ such that:
\begin{itemize}
 \item $\tX|D(\Lambda^{n+1}[n+1]) = X$ and
 \item $\S_A \tX \sim_{E[1]} X'$ relative to the boundary.
\end{itemize}
By \Cref{thm:Sd-D-acyclic}, it suffices to define $\tX \colon D(\Lambda^{n+1} [n+1])^\op \cup \Sd[n+1]^\op \to \C$. Define $\tX = X$ on $D(\Lambda^{n+1} [n+1])^\op$ and $\tX_{[n]} = \tX_{[n+1]} = X'_{[n+1]}$.

This gives a homotopical diagram  $\tX \colon D(\Lambda^{n+1} [n+1])^\op \cup \Sd[n+1]^\op \to \C$, which, by \Cref{lem:extend-reedy}, we may fibrantly replace without changing its value on $D(\Lambda^{n+1} [n+1])^\op$. Applying \Cref{thm:Sd-D-acyclic}, we extend $\tX$ to $D[n+1]^\op$, as required.

It is clear that the upper triangle in the diagram above commutes. The lower triangle commutes up to $E[1]$-homotopy relative to the boundary by \Cref{lem:rel-homot}.
\end{proof}

\section{Adjoints between quasicategories of frames} \label{sec:adjoints}

The main goal of this section is the proof of the following theorem:

\begin{theorem}\label{thm:Nf-preserves-adjoints}
 Let $F \colon \C \rightleftarrows \D : \! G$ be an adjoint pair of exact functors. Then $\Nf F \colon \Nf\C \rightleftarrows \Nf\D :\! \Nf G$ is an adjunction of quasicategories.
\end{theorem}

Recall that by \Cref{lem:adjoint-characterization}, a simplicial map $G \colon \D \to \C$ of quasicategories is a right adjoint if for any $x \in \C$, the quasicategory $x \slice G$ has an initial object (the \emph{unit} of the adjunction). We will therefore begin with the construction of a candidate unit of $\Nf F \adjoint \Nf G$.

\begin{construction}[Unit of $\Nf F \adjoint \Nf G$] \label{unit}
 Let $A \colon D[0]^\op \to \C$ be $0$-simplex in $\Nf\C$. We will construct a $1$-simplex $X \colon D[1]^\op \to \C$ in $A \slice \Nf G$ (thus, we must have $X|D[0] = A$). Since both $F$ and $G$ are exact, we put $X_{1^k} := GFA_{0^k}$.
 
 This defines $X$ on $D\partial \Delta[1]^\op$. By \Cref{thm:Sd-D-acyclic}, it suffices to extend it further to $\Sd[1]^\op$, that is, to find $X_1$ and a Reedy fibrant fraction $X_0 \leftarrow X_{01} \to X_1$. Define a factorization of the map $X_0 \to X_0 \times X_1$ as a weak equivalence followed by a fibration:
  \begin{ctikzpicture}
   \matrix[diagram]
   {
     |(a)| X_0 &              & |(b)| X_0 \times X_1 \\
               & |(c)| X_{01} &  \\
   };

   \draw[->] (a) to node[above] {$\langle 1_{X_0} \eta_{X_0}\rangle$} (b);
   \draw[->] (a) to node[below left] {$w$} node[above right] {$\we$} (c);
   \draw[->>] (c) to node[below right] {$\langle p_0, p_1 \rangle$} (b);
 \end{ctikzpicture}
This gives the extension to $\Sd[1]^\op \cup D\partial \Delta[1]^\op$, which in turn yields $X \colon D[1]^\op \to \C$ as required. \qed
\end{construction}

\begin{proposition}\label{prop:unit-is-initial}
 The $1$-simplex $X \colon D[1]^\op \to \C$ constructed above is initial in $A \slice \Nf G$.
\end{proposition}

Our next goal is the proof of \Cref{prop:unit-is-initial}. Before doing so, we shall prove an auxiliary lemma that will allow us to take advantage of the adjunction $F \adjoint G$ later in our proof.

\begin{lemma}\label{lem:adjoints-preserve-homotopy}
 Let $F \colon  \C \rightleftarrows \D :\! G$ be an adjoint pair of exact functors. If two morphism $f,g \colon A \to GB$ are homotopic in $\C$, then their adjoint transposes $\overline{f}, \overline{g} \colon FA \to B$ are homotopic in $\D$.
\end{lemma}

\begin{proof}
Choose a path object $B \overset\we\to PB \fib B \times B$ for $B$ in $\D$. By assumption and exactness of $G$, there exists a commutative square of the form:
 \begin{ctikzpicture}
   \matrix[diagram]
   {
     |(a)| \widetilde{A} & & |(b)| GPB \\
     |(c)| B & & |(d)| GB \times GB \\
   };

   \draw[->] (a) to (b);
   \draw[->>] (a) to node[left] {$\we$} (c);
   \draw[->] (b) to (d);
   \draw[->] (c) to node[above] {$\langle f, g \rangle$} (d);
 \end{ctikzpicture}
 By adjointness and exactness of $F$, we get:
 \begin{ctikzpicture}
   \matrix[diagram]
   {
     |(a)| F\widetilde{A} & & |(b)| PB \\
     |(c)| FA & & |(d)| B \times B \\
   };

   \draw[->] (a) to (b);
   \draw[->>] (a) to node[left] {$\we$} (c);
   \draw[->] (b) to (d);
   \draw[->] (c) to node[above] {$\langle \overline{f}, \overline{g} \rangle$} (d);
 \end{ctikzpicture}
 and thus $\overline{f}$ and $\overline{g}$ are homotopic.
\end{proof}

A $1$-simplex $X \colon A \to GB$ is initial in $A \slice \Nf G$ if for all $n > 0$ and any map $Y \colon \Delta[0]\star \partial \Delta[n] \to A \slice \Nf G$ such that $Y|\Delta[0] \star \Delta^{\{0\}} = X$, there exists an extension:
 \begin{ctikzpicture}
   \matrix[diagram]
   {
     |(a)| \Delta[0]\star\partial \Delta [n] & |(c)| A \downarrow \Nf G \\
     |(b)| \Delta[0]\star\Delta [n] &  \\
   };

   \draw[->] (a) to node[above] {$Y$} (c);
   \draw[->,inj] (a) to (b);
   \draw[->,dashed] (b) to node[below right] {$\tY$} (c);
 \end{ctikzpicture}

Before giving the proof of \Cref{prop:unit-is-initial}, we consider the case $n=1$ separately to build the intuition for the general case.

\begin{lemma}\label{lem:unit-for-1}
 Let $X \colon D[1]^\op \to \C$ be the diagram of \Cref{unit}. Then every diagram $Y \colon \Delta[0] \star \partial \Delta[1] \to \C$ such that $Y|\Delta[0] \star \Delta^{\{0\}} = X$ admits an extension:
 \begin{ctikzpicture}
   \matrix[diagram]
   {
     |(a)| \Delta[0]\star\partial \Delta [1] & |(c)| A \downarrow \Nf G \\
     |(b)| \Delta[0]\star\Delta [1] &  \\
   };

   \draw[->] (a) to node[above] {$Y$} (c);
   \draw[->,inj] (a) to (b);
   \draw[->,dashed] (b) to node[below right] {$\tY$} (c);
 \end{ctikzpicture}
\end{lemma}

We will write $(n]$ for the linearly ordered set $\{1 \leq 2 \leq \ldots \leq n\} \subseteq [n]$. In particular, we have $\N(n] \iso \Delta[n-1]$, but the vertices of $\N(n]$ are numbered with $1$, $2$, \ldots, $n$.

\begin{proof}
 Unwinding the definitions, we see that we are given:
\begin{itemize}
 \item $Y \colon D\Lambda^0[2]^\op \to \C$ with $Y|D(\Delta^{\{0,1\}})^\op = X$;
 \item $B \colon D\partial\N(2]^\op \to \D$ with $B_{1^k} = FX_{0^k}$ and $Y|(D\partial\N(2]^\op) = GB$,
\end{itemize} 
 and we are seeking Reedy fibrant diagrams $\tB \colon D(2]^\op \to \D$ and $\tY \colon D[2]^\op \to \C$ with $\tY|D(2]^\op = G\tB$.
 
 By \Cref{thm:Sd-D-acyclic}, it suffices to find extensions $\tB \colon \Sd(2]^\op \to \D$ and $\tY \colon \Sd[2]^\op \to \C$ with $\tY|\Sd(2]^\op = G\tB$. Explicitly, we need to find:
 \begin{itemize}
  \item $\tB_{12} \in \D$ together with Reedy fibrant fraction $B_1 \overset\we\twoheadleftarrow \tB_{12} \fib B_2$;
  \item $\tY_{012}$ fitting into a Reedy fibrant diagram:
   \[
\begin{tikzpicture}
  \matrix[diagram]
  {
    & & |(b)| X_1 \\
    \\ \\
    |(a)| X_0 & & & & |(c)| GB_2 \\
  };
  \node (ab)  at (barycentric cs:a=1,b=1,c=0) {$X_{01}$};
  \node (bc)  at (barycentric cs:a=0,b=1,c=1) {$G\tB_{12}$};
  \node (ac)  at (barycentric cs:a=1,b=0,c=1) {$Y_{02}$};
  \node (abc) at (barycentric cs:a=1,b=1,c=1) {$\tY_{012}$};
  \draw[->] (ab) to node[above left] {$\we$} (a);
  \draw[->] (ab) to (b);
  \draw[->] (bc) to node[above right] {$\we$} (b);
  \draw[->] (bc) to (c);
  \draw[->] (ac) to node[above] {$\we$} (a);
  \draw[->] (ac) to (c);
  \draw[->] (abc) to node[above] {$\we$} (ab);
  \draw[->] (abc) to node[right] {$\we$} (ac);
  \draw[->] (abc) to (bc);
\end{tikzpicture}
\]
 \end{itemize}
 
Define $\tY_{012}$ as the pullback:
 \begin{ctikzpicture}
   \matrix[diagram]
   {
     |(a)| \tY_{012} & |(b)| X_{01} \\
     |(c)| Y_{02} & |(d)| X_0 \\
   };

   \draw[->] (a) to node[above] {$\pi_1$} (b);
   \draw[->] (a) to node[left] {$\pi_2$} (c);
   \draw[->] (b) to node[right] {$p_0$} (d);
   \draw[->] (c) to (d);
 \end{ctikzpicture}
and let $\tB_{12} := F\tY_{012}$. By the universal property of $\eta$, there exists a unique map $\langle f_1 , f_2 \rangle \colon \tB_{12} \to B_1 \times B_2$ making the following square commute:
 \begin{ctikzpicture}
   \matrix[diagram]
   {
     |(a)| \tY_{012} & & |(b)| G\tB_{12} \\
     |(c)| X_{01} \times Y_{02} & & |(d)| G(B_1 \times B_2) \\
   };

   \draw[->] (a) to node[above] {$\eta_{\tY_{012}}$} (b);
   \draw[->] (a) to node[left] {$\we$} (c);
   \draw[->,dashed] (b) to node[right] {$G\langle f_1, f_2 \rangle$} (d);
   \draw[->] (c) to (d);
 \end{ctikzpicture}
We claim that up to Reedy fibrancy of the resulting diagrams, the objects $\tY_{012}$ and $\tB_{12}$ with the maps defined above give the extension required in the universal property. To see this, it remains to verify that the map $f_1 \colon \tB_{01} \to B_1$ is a weak equivalence. By \Cref{lem:extend-reedy}, we can then fibrantly replace the diagrams, obtaining the desired ones.

The map $p_0\pi_1 \colon \tY_{012} \to X_0$ is weak equivalence as a composite of two weak equivalences, and consequently $F(p_0\pi_1) \colon \tB_{01} \to B_1$ is a weak equivalence. Since a map homotopic to a weak equivalence is itself a weak equivalence, it suffices to show that $f_1$ and $F(p_0\pi_1)$ are homotopic.

By \Cref{lem:adjoints-preserve-homotopy}, it is enough to show that their adjoint transposes $\overline{F(p_0 \pi_1)}, \overline{f_1} \colon \tY_{012} \to GB_1$ are homotopic. We will do so by explicitly computing these transposes.

By naturality of $\eta$, the following diagram commutes:
 \begin{ctikzpicture}
   \matrix[diagram]
   {
     |(a)| \tY_{012} & & |(b)| GF\tY_{012} \\
     |(c)| X_0 & & |(d)| GFX_0 \\
   };

   \draw[->] (a) to node[above] {$\eta_{\tY_{012}}$} (b);
   \draw[->] (a) to node[left] {$p_0\pi_1$} (c);
   \draw[->] (b) to node[right] {$GF(p_0\pi_1)$} (d);
   \draw[->] (c) to node[above] {$\eta_{X_0}$} (d);
 \end{ctikzpicture}
and hence $\overline{F(p_0 \pi_1)} = \eta_{X_0}p_0\pi_1$.

On the other hand, the following diagram commutes by the definition of $\tB_{12}$:
 \begin{ctikzpicture}
   \matrix[diagram]
   {
     |(a)| \tY_{012} & & |(b)| GF\tY_{012} \\
     |(c)| X_{01} \times Y_{02} & & |(d)| GB_1 \times GB_2 \\
     |(e)| X_{01} & & |(f)| GFX_0 \\
   };

   \draw[->] (a) to node[above] {$\eta_{\tY_{012}}$} (b);
   \draw[->] (a) to (c);
   \draw[->] (b) to (d);
   \draw[->] (c) to (d);
   \draw[->] (c) to (e);
   \draw[->] (d) to (f);
   \draw[->] (e) to node[above] {$p_1$} (f);
   \draw[->] (a) to [bend right=70] node[left] {$\pi_1$} (e);
   \draw[->] (b) to [bend left=70] node[right] {$Gf_1$} (f);
 \end{ctikzpicture}

and hence $\overline{f_1} = p_1 \pi_1$.

Thus, we need to show that the maps $\eta_{X_0}p_0\pi_1$ and $p_1 \pi_1$ are homotopic. However, since $\eta_{X_0}p_0 w = p_1 w$ by \Cref{unit} and $w$ is a weak equivalence, we obtain that $\eta_{X_0}p_0$ and $p_1$ are homotopic, and hence $\eta_{X_0}p_0 \pi_1$ and $p_1\pi_1$ must also be homotopic.
\end{proof}

We are now ready to prove the general case.

\begin{proof}[Proof of \Cref{prop:unit-is-initial}]
We are given:
\begin{itemize}
 \item $Y \colon D\Lambda^0[1+n]^\op \to \C$ with $Y|D(\Delta^{\{0,1\}})^\op = X$;
 \item $B \colon D\partial\N(1+n]^\op \to \D$ with $B_{1^k} = FX_{0^k}$ and $Y|(D\partial\N(1+n]^\op) = GB$,
\end{itemize} 
 and we are seeking Reedy fibrant diagrams $\tB \colon D(1+n]^\op \to \D$ and $\tY \colon D[1+n]^\op \to \C$ with $\tY|D(1+n]^\op = G\tB$.
 
 By \Cref{thm:Sd-D-acyclic}, it suffices to find extensions $\tB \colon \Sd(1+n]^\op \to \D$ and $\tY \colon \Sd[1+n]^\op \to \C$ with $\tY|\Sd(1+n]^\op = G\tB$.
 
We start by defining $\tY_{[1+n]}$ as a limit:
 \[ \tY_{[1+n]} = \lim_{0 \in S \subsetneq [1+n]} Y_S.\]
Moreover, we set $\tB_{(1+n]} := F\tY_{[1+n]}$
By construction, the canonical map $\tY_{[1+n]} \to Y_S$ is a fibration for any $S \subsetneq [1+n]$ such that $0 \in S$. By \cite[Lem.\ 2.17]{kapulkin-szumilo}, it is moreover an acyclic fibration.

Given $S \subseteq [n]$ such that $0 \not \in S$, there is a map $Y_{S \cup \{0\}} \to Y_S = GB_S$ yielding:
 \[ \lim_{0\not \in S} Y_{S \cup \{0\}} \longrightarrow \lim_{0 \not \in S} GB_S \iso G(\lim_{0 \not \in S} B_S)\]
 since $G$ is a right adjoint.

By the universal property of $\eta$, there is a unique morphism
 \[ \tB_{(1+n]} \longrightarrow \lim_{0 \not \in S} B_S\]
 making the following square commute (recall that $\tB_{(1+n]} = F\tY_{[1+n]}$):
  \begin{ctikzpicture}
   \matrix[diagram]
   {
     |(a)| \tY_{[1+n]} & & |(b)| GF\tY_{[1+n]} \\
     |(c)| \lim\limits_{0 \not \in S} Y_{S \cup \{0\}} & 
                              & |(d)| G(\lim\limits_{0 \not \in S} B_S) \\
   };

   \draw[->] (a) to node[above] {$\eta_{\tY_{[1+n]}}$} (b);
   \draw[->] (a) to (c);
   \draw[->] (b) to (d);
   \draw[->] (c) to (d);
 \end{ctikzpicture}
 We claim that $\tY_{[1+n]}$ and $\tB_{(1+n]}$ together with the map constructed above give the desired extension. For that, it remains to verify that the canonical map $\tB_{(1+n]} \to B_{1}$ is a weak equivalence, but the proof of this fact follows now verbatim the proof of \Cref{lem:unit-for-1}.
\end{proof}

\begin{proof}[Proof of \Cref{thm:Nf-preserves-adjoints}]
 By \Cref{prop:unit-is-initial,lem:adjoint-characterization}, $\Nf G$ is a right adjoint and by \Cref{unit} its left adjoint is given by $\Nf F$.
\end{proof}

\section{Quasicategories arising from type theory are locally cartesian closed}
\label{sec:the-proof}

In this section, we complete the proof of \Cref{thm:main-first-statement}. We begin by introducing the notion of a locally cartesian closed fibration category (\Cref{def:lcc-fib-cat}) and, using the results of the previous sections, show that the quasicategory of frames of such a category is locally cartesian closed (\Cref{thm:lccfib-gives-lccqcat}). We then verify that every contextual category carries the structure of  a locally cartesian closed fibration category, which lets us deduce our main theorem (\Cref{thm:main-second-statement}).

\begin{definition} \label{def:lcc-fib-cat}
 A \textbf{locally cartesian closed fibration category} is a fibration category $\C$ such that for any fibration $p \colon A \to B$, the pullback functor $p^* \colon \C\fibslice B \to \C \fibslice A$ admits a right adjoint which is an exact functor.
\end{definition}

We will show that if $\C$ is a locally cartesian closed fibration category, then $\Nf\C$ is a locally cartesian closed quasicategory. The following lemma contains the technical heart of the proof.

\begin{lemma}\label{lem:comm-diagram}
 Let $\C$ be a fibration category and $X \colon A \to B$ a $1$-simplex in $\Nf\C$ (thus, in particular, $X_0 = A_0$ and $X_1 = B_0$). Let $f$ denote the acyclic fibration $X_{01} \to A_0$ and by $g$ the fibration $X_{01} \to B_0$. Then the diagram:
\begin{ctikzpicture}
   \matrix[diagram]
   {
     |(a)| \Nf\C\slice B & & |(b)| \Nf(\C\fibslice B_0) \\
     |(c)|  \Nf\C\slice A & |(c')|\Nf(\C\fibslice A_0) & 
                                           |(d)| \Nf(\C\fibslice X_{01}) \\
   };

   \draw[->] (a) to node[above] {$\S_B$} (b);
   \draw[->] (a) to node[left] {$X^*$} (c);
   \draw[->] (b) to node[right] {$\Nf g^*$} (d);
   \draw[->] (c) to node[above] {$\S_A$} (c');
   \draw[->] (c') to node[above] {$\Nf f^*$} (d);
 \end{ctikzpicture}
 commutes up to $E[1]$-homotopy.
\end{lemma}

We will write $\sqcup$ for the simplicial set with four $0$-simplices: $(0,0)$, $(0,1)$, $(1,0)$, and $(1,1)$; and three non-degenerate $1$-simplices: $(0,0) \to (0,1)$, $(0,1) \to (1,1)$, and $(1,0) \to (1,1)$.

\begin{proof}
 The map $\Nf f^*$ is a categorical equivalence between quasicategories and thus admits a homotopy inverse, given by $\Nf f_!$ (which is obtained by choosing a section of the projection $(\C \fibslice A_0)_{f_!} \to (\C \fibslice A_0)$ as in \Cref{rmk:Nf-on-homotopical}). By \Cref{thm:slices}, the functor $\S_A$ is a categorical equivalence and thus admits a homotopy inverse $\overline{\S}_A$.

 Let $F = \overline{\S}_A (\Nf f_!) (\Nf g^*) \S_B$. Given any $Y \in \Nf\C \slice B$, we will define a diagram $[FY, X, Y] \colon (D \sqcup)^\op \to \C$ as follows: since $D$ preserves colimits \cite[Lem.\ 3.6]{szumilo:two-models}, it suffices to give three compatible diagrams $D[1]^\op \to \C$, namely:
 \begin{ctikzpicture}
   \matrix[diagram]
   {
     |(a)| (0,0) & |(b)| (1,0) \\
     |(c)| (0,1) & |(d)| (1,1) \\
   };

   \draw[->] (a) to node[left] {$FY$} (c);
   \draw[->] (b) to node[right] {$Y$} (d);
   \draw[->] (c) to node[above] {$X$} (d);
 \end{ctikzpicture}
  By the universal property of the pullback in the quasicategory $\Nf\C$, it suffices to show that $[FY, X, Y]$ can be extended to a pullback square in $\Nf\C$, i.e.\ a diagram $D([1] \times [1])^\op \to \C$ satisfying the condition of \Cref{ex:pullback-in-NfC}. Indeed, such an extension defines a simplicial map $(\Nf\C \slice B)_0 \to (\Nf\C)^{(\Delta[1] \times \Delta[1])_X}_{\mathrm{univ}}$, where $(\Nf\C \slice B)_0$ denotes the $0$-skeleton of $\Nf\C \slice B$, fitting into a commutative square:
   \begin{ctikzpicture}
   \matrix[diagram]
   {
     |(a)|(\Nf\C \slice B)_0 & |(b)| (\Nf\C)^{(\Delta[1] \times \Delta[1])_X}_{\mathrm{univ}} & |(x)| \Nf\C\slice A \\
     |(c)| \Nf\C \slice B   & |(d)| \Nf\C \slice B &  \\
   };
   \draw[->] (a) to (b);
   \draw[->,inj] (a) to (c);
   \draw[->>] (b) to node[right] {$\we$} (d);
   \draw[->] (c) to node[below] {$=$} (d);
   \draw[->,dashed] (c) to (b);
   \draw[->] (b) to (x);
 \end{ctikzpicture}
 which establishes $F$ as a particular choice of the pullback functor $\Nf\C \slice B \to \Nf\C\slice A$.
 
 Let $Q = g^*Y_{01}$. Factor the composite $Q \to X_{01} \overset{f}{\to} A_0$ into a weak equivalence followed by a fibration: $Q \overset{\sim}{\to} Q_1 \fib A_0$ (i.e.\ $Q_1$ should be thought of as the value of $(\Nf f_!) (\Nf g^*) \S_B$ on $Y_{01}$). Since $\S_A$ and $\overline{\S}_A$ are homotopy inverses, we obtain a fraction $FY_{01} \overset{\sim}{\twoheadleftarrow} Q_2 \overset{\sim}{\fib} Q_1$ in $\C \fibslice A_0$. Pulling back the diagram:
 \begin{ctikzpicture}
   \matrix[diagram]
   {
     |(a)| FY_{01} & |(b)| Q_2 & |(c)| Q_1 \\
                   & |(d)| A_0 &  \\
   };

   \draw[->>] (b) to node[above] {$\we$} (a);
   \draw[->>] (b) to node[above] {$\we$} (c);
   \draw[->>] (a) to (d);
   \draw[->>] (b) to (d);
   \draw[->>] (c) to (d);
 \end{ctikzpicture}
 along $f \colon X_{01} \to A_0$, we obtain a zigzag 
 \begin{ctikzpicture}
   \matrix[diagram]
   {
     |(a)| f^*FY_{01} & |(b)| f^*Q_2 & |(c)| f^*Q_1 & |(d)| Q \\
   };

   \draw[->>] (b) to node[above] {$\we$} (a);
   \draw[->>] (b) to node[above] {$\we$} (c);
   \draw[->] (d) to node[above] {$\we$} (c);
 \end{ctikzpicture}
 in $\C\fibslice X_{01}$. Define $Q' = f^*Q_2 \times_{f^*Q_1} Q$. Then:
 \begin{center}
 \begin{ctikzpicture}
    \matrix[diagram,column sep=10em,row sep=10em]
    {
      |(a-a)| FY_1 & |(b-a)| Y_1 \\
      |(a-b)| A_0 & |(b-b)| B_0 \\
    };

    \node (aa-ab) at (barycentric cs:a-a=1,a-b=1,b-b=0) {$FY_{01}$};
    \node (ab-bb) at (barycentric cs:a-a=0,a-b=1,b-b=1) {$X_{01}$};
    \node (ab-ab) at (barycentric cs:a-a=1,a-b=0,b-b=1) {$Q'$};

    \node (ab-aa) at (barycentric cs:a-a=1,b-a=1,b-b=0) {$Q'$};
    \node (bb-ab) at (barycentric cs:a-a=0,b-a=1,b-b=1) {$Y_{01}$};

    \node (aab-abb) at (barycentric cs:a-a=2,a-b=3,b-b=2) {$Q'$};
    \node (abb-aab) at (barycentric cs:a-a=2,b-a=3,b-b=2) {$Q'$};

    \draw[->]  (aa-ab) to node[left] {$\we$} (a-a);
    \draw[->] (aa-ab) to (a-b);

    \draw[->] (ab-bb) to node[below] {$\we$} node[above] {$f$} (a-b);
    \draw[->] (ab-bb) to node[above] {$g$} (b-b);

    \draw[->] (ab-aa) to node[above] {$\we$} (a-a);
    \draw[->] (ab-aa) to (b-a);

    \draw[->] (bb-ab) to node[right] {$\we$} (b-a);
    \draw[->] (bb-ab) to (b-b);

    \draw[->] (ab-ab) to node[above right] {$\we$} (a-a);
    \draw[->,dashed] (ab-ab) to (b-b);

    \draw[->] (aab-abb) to node[above right] {$\we$} (aa-ab);
    \draw[->,dashed] (aab-abb) to (ab-bb);
    \draw[->] (aab-abb) to node[below right]  {$\we$} (ab-ab);

    \draw[->] (abb-aab) to node[below left] {$\we$} (ab-aa);
    \draw[->,dashed] (abb-aab) to  (bb-ab);
    \draw[->] (abb-aab) to node[below right] {$\we$} (ab-ab);
  \end{ctikzpicture}
 \end{center} 
 defines a homotopical diagram $ P \colon \Sd([1] \times [1])^\op \to \C$ whose restriction to $(\Sd \sqcup)^\op$ is $[FY, X, Y]$. Moreover, by construction, the canonical map induced by the dashed arrows $Q' = Q' \times_{Q'} Q' \to X_{01} \times_{B_0} Y_{01} = Q$ is a weak equivalence.
 
 By \Cref{lem:extend-reedy}, we may find a Reedy fibrant diagram $\widetilde{P} \colon \Sd([1] \times [1])^\op \to \C$ with a weak equivalence $P \overset\we\to \widetilde{P}$ such that $P|(\Sd \sqcup)^\op = [FY, X, Y]$. Finally, by \Cref{thm:Sd-D-acyclic}, one can extend $\widetilde{P}$ to a diagram $D([1] \times [1])^\op \to \C$ satisfying the pullback condition of \Cref{ex:pullback-in-NfC}, thus completing the proof. 
\end{proof}

\begin{theorem}\label{thm:lccfib-gives-lccqcat}
 Let $\C$ be a locally cartesian closed fibration category. Then the quasicategory $\Nf\C$ is locally cartesian closed.
\end{theorem}

\begin{proof}
 By \Cref{thm:szumilo-main}, $\Nf\C$ has finite limits, so we only need to show that for any $1$-simplex $X \colon A \to B$ in $\Nf\C$, the pullback functor $X^*$ has a right adjoint. However, in the diagram of \Cref{lem:comm-diagram}, the horizontal maps are equivalences by \Cref{thm:slices,prop:pb-along-we} and the right vertical map admits a right adjoint by assumption and \Cref{thm:Nf-preserves-adjoints}.
\end{proof}

Next, we shall verify that the fibration categories given by categorical models of type theory are locally cartesian closed in the sense of \Cref{def:lcc-fib-cat}.

\begin{proposition}\label{prop:ClT-is-lccfib}
 If $\mC$ is a categorical model of type theory, then its underlying fibration category of \Cref{thm:akl} is locally cartesian closed in the sense of \Cref{def:lcc-fib-cat}.
\end{proposition}

The proof of this proposition will be preceded by two lemmas. We begin by introducing some notation. Let $\Gamma$ be an object of a categorical model of type theory $\C$ and consider two context extensions $\Gamma.\Delta$ and $\Gamma.\Theta$. There is an evident map:
\[\mathsf{exch}_{\Delta,\Theta} \colon \Gamma.\Delta.p_{\Delta}^*\Theta \to \Gamma.\Theta.p_\Theta^*\Delta \]
given by the universal property of a pullback.

\begin{lemma}\label{lem:right-adj-in-ClT}
 For any dependent projection $p_\Delta \colon (\Gamma.\Delta) \to \Gamma$ in a categorical model of type theory $\mC$, the pullback functor $p_\Delta^* \colon \mC\fibslice \Gamma \to \mC \fibslice (\Gamma.\Delta)$ admits a right adjoint.
\end{lemma}

\begin{proof}
 Let $\Gamma.\Delta.\Theta \in \mC\fibslice (\Gamma.\Delta)$ and define $(p_\Delta)_* (\Gamma.\Delta.\Theta) := \Gamma.\mathsf{\Pi}(\Delta,\Theta)$. The counit
 \[\varepsilon_{\Gamma.\Delta.\Theta} \colon \Gamma.\Delta.p_\Delta^*\mathsf{\Pi}(\Delta, \Theta) \to \Gamma.\Delta.\Theta \]
 is given by the composite of the $\mathsf{exch}_{\Delta, \mathsf{\Pi}(\Delta,\Theta)}$ followed by the map $\mathsf{app}_{\Delta,\Theta}$ of \cite[Def.\ B.1.1]{kapulkin-lumsdaine-voevodsky:simplicial-model}. The universal property follows by $\mathsf{\Pi}$-$\eta$.
\end{proof}

\begin{lemma}\label{lem:def-ac}
 Consider an iterated context extension $\Gamma.\Delta.\Theta.\Psi$ in a categorical model of type theory $\mC$. Then the contexts:
 \[ \Gamma.\mathsf{\Pi}(\Delta, \Theta).\mathsf{\Pi}(p_{\mathsf{\Pi}(\Delta,\Theta)}^* \Delta.\mathsf{app}_{\Delta, \Theta}^* \Psi) \quad \text{ and } \quad \Gamma.\mathsf{\Pi}(\Delta, \Theta.\Psi) \]
 are isomorphic (and, in fact, equal) in $\mC$.
\end{lemma}

\begin{proof}
 That follows immediately by the construction of the $\mathsf{\Pi}$-structure on $\mC^\cxt$.
\end{proof}

\begin{proof}[Proof of Proposition \ref{prop:ClT-is-lccfib}]
By \Cref{lem:right-adj-in-ClT}, for every dependent projection $p_A \colon (\Gamma. A) \to \Gamma$, the pullback functor $p_A^*$ has a right adjoint $(p_A)_*$. As a right adjoint, $(p_A)_*$ preserves finite limits, thus, in particular, pullbacks and the terminal object. Moreover, by \Cref{lem:def-ac}, it preserves fibrations and, by easy calculation using Functional Extensionality, weak equivalences.
\end{proof}

\begin{theorem} \label{thm:NfClT-is-lcc}
 Given a categorical model of type theory $\mC$, the quasicategory $\Nf \mC$ is locally cartesian closed.
\end{theorem}

\begin{proof}
 By \Cref{prop:ClT-is-lccfib}, $\mC$ is a locally cartesian closed fibration category and hence, by \Cref{thm:lccfib-gives-lccqcat}, $\Nf\mC$ is a locally cartesian closed quasicategory.
\end{proof}

We are now ready to prove our main theorem:

\begin{theorem} \label{thm:main-second-statement}
 For any simplicial localization functor $\Ho_\infty \colon \weCat \to \sSet$ and any categorical model of type theory, the quasicategory $\Ho_\infty \U \mC$ is locally cartesian closed.
\end{theorem}

\begin{proof}
 By \Cref{thm:Ho-Nf-coincide}, for any fibration category $\C$, there is an equivalence of quasicategories $\Nf\C$ and $\Ho_\infty\C$. Thus, since $\Nf\mC$ is locally cartesian closed by \Cref{thm:NfClT-is-lcc}, so is $\Ho_\infty\U \C$.
\end{proof}

\begin{remark}
Every homomorphism of models of type theory $F \colon \mC \to \mathsf{D}$ induces a finite-limit preserving functor $\Nf \U F$ of the associated quasicategories.
\end{remark}

\bibliographystyle{amsalphaurlmod}
\bibliography{general-bibliography}

\end{document}